\documentclass[12pt,bezier]{article}
\usepackage{amssymb}
\usepackage{mathrsfs}
\usepackage{amsmath}
\usepackage{amsfonts,amsthm,amssymb}
\usepackage{amsfonts}
\usepackage{graphicx}
\usepackage{subfigure}
\usepackage[]{caption2}
\usepackage{color}
\usepackage{cite}
\usepackage{hyperref}
\newtheorem{lem}{Lemma}
\newtheorem{thm}{Theorem}

\renewcommand{\thethm}{\Alph{thm}}

\begin{document}
\title{Contractible Non-Edges in 3-Connected Graphs}

\author{Shuai Kou$^a$, \quad Chengfu Qin$^b$, \quad Weihua Yang$^a$\footnote{Corresponding author. E-mail: yangweihua@tyut.edu.cn, ywh222@163.com}, \quad Mingzu Zhang$^c$\\
\small $^a$ Department of Mathematics, Taiyuan University of Technology, Taiyuan 030024, China\\
\small $^b$ Department of Mathematics, Nanning Normal University, Nanning 530001, China\\
\small $^c$ College of Mathematics and System Sciences, Xinjiang University, Urumqi 830046, China\\
}
\date{}

\maketitle {\flushleft\bf Abstract:} {\small  We call a pair of non-adjacent vertices in $G$ a ``non-edge''. Contraction of a non-edge $\{u, v\}$ in $G$ is the replacement of $u$ and $v$ with a single vertex $z$ and then making all the vertices that are adjacent to $u$ or $v$ adjacent to $z$. A non-edge $\{u, v\}$ is said to be contractible in a $k$-connected graph $G$, if the resulting graph after its contraction remains $k$-connected. Tsz Lung Chan characterized all 3-connected graphs (finite or infinite) that does not contain any contractible non-edges in 2019, and posed the problem of characterizing all 3-connected graphs that contain exactly one contractible non-edge. In this paper, we solve this problem.}


\section{Introduction}
In this paper, we only consider simple undirected graphs, with undefined terms and notations following~\cite{Bondy1}. For a graph $G$, let $V(G)$ and $E(G)$ denote the set of vertices of $G$ and the set of edges of $G$, respectively. The set of neighbours of a vertex $x$ is denoted by $N_{G}(x)$. We define the degree of $x\in V(G)$ by $d_{G}(x)$, namely $d_{G}(x)=|N_{G}(x)|$. For $S\subseteq V(G)$, let $G[S]$ denote the subgraph induced by $S$, and let $G-S$ denote the graph obtained from $G$ by deleting the vertices of $S$ together with the edges incident with them. A \emph{wheel} $W_{n}$ of order $n$ is a graph obtained from a cycle of order $n-1$ by adding a new vertex adjacent to all the vertices of the cycle. The complete graph of order $n$ is denoted by $K_{n}$. Let $K_{n}^{-}$ be the graph obtained from $K_{n}$ by deleting one edge. If vertices $u$ and $v$ are connected in $G$, the \emph{distance} between $u$ and $v$ in $G$, denoted by $d_{G}(u, v)$, is the length of a shortest $(u, v)$-path in $G$.

An edge $e=xy$ of $G$ is said to be \emph{contracted} if it is deleted and its ends are identified. Let $k$ be an integer such that $k\geq 2$ and let $G$ be a $k$-connected graph with $|V(G)|\geq k+2$. An edge $e$ of $G$ is said to be \emph{k-contractible} if the contraction of the edge results in a $k$-connected graph. A \emph{non-edge} of $G$ is a pair of non-adjacent vertices in $G$. Like edge contractions, contraction of a non-edge $\{u, v\}$ in $G$ is the replacement of $u$ and $v$ with a single vertex $z$ and then making all the vertices that are adjacent to $u$ or $v$ adjacent to $z$. A non-edge $\{u, v\}$ is said to be contractible in a $k$-connected graph, if the resulting graph after its contraction remains $k$-connected.

There have been many results concerning the number and distribution of contractible edges of 3-connected graphs. Kriesell~\cite{Kriesell1} studied contractible non-edges of 3-connected finite graphs and proved the following result.

\begin{thm}\cite{Kriesell1}\label{thmA}
Every non-complete 3-connected finite graph neither isomorphic to a wheel nor isomorphic to one of the ten graphs in Fig.~\ref{fig1} contains a contractible non-edge.
\end{thm}

\begin{figure}
  \centering
  \includegraphics[width=15cm]{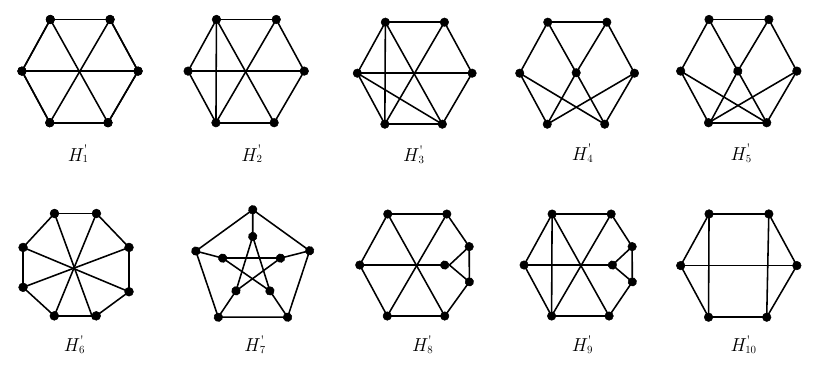}\\
  \caption{3-connected graphs other than complete graphs and wheels that do not contain any contractible
           non-edges}
  \label{fig1}
\end{figure}

Tsz Lung Chan~\cite{Tsz2} extended Theorem~\ref{thmA} and characterized all 3-connected graphs (finite or infinite) that does not contain any contractible non-edges. Furthermore, he derived a tight upper bound for the number of contractible non-edges.

\begin{thm}\cite{Tsz2}\label{thmB}
Every non-complete 3-connected graph (finite or infinite) neither isomorphic to a wheel nor isomorphic to one of the ten graphs in Fig.~\ref{fig1} contains a contractible non-edge. In particular, every non-complete 3-connected infinite graph contains a contractible non-edge.
\end{thm}

\begin{thm}\cite{Tsz2}\label{thmC}
Let $G$ be a 3-connected graph of order $n$ where $n\geq6$. Then the number of contractible non-edges is at most $\frac{n(n-5)}{2}$. All the extremal graphs with at least seven vertices are characterized to be 4-connected 4-regular.
\end{thm}

Some results on 2-connected graphs and $k$-connected triangle-free graphs can be found in~\cite{Das, Kriesell2}. In~\cite{Tsz2}, Chan posed the problem of characterizing all 3-connected graphs that contain exactly one contractible non-edge. In this paper, we resolve this problem and obtain the following result.

\setcounter{thm}{0}
\renewcommand{\thethm}{\arabic{thm}}
\begin{thm}\label{thm1}
Let $G$ be a 3-connected graph. Then $G$ has exactly one contractible non-edge if and only if $G$ is isomorphic to $K_{n}^{-}$ where $n\geq 5$, or one of the four graphs in Fig.~\ref{fig2}.
\end{thm}

\begin{figure}
  \centering
  \includegraphics{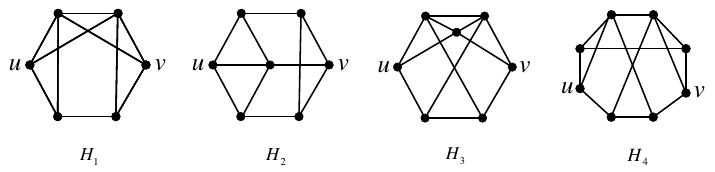}\\
  \caption{3-connected graphs other than $K_{n}^{-}(n\geq5)$ that contain exactly one contractible non-edge $\{u, v\}$}
  \label{fig2}
\end{figure}

\section{Preliminaries}
In this section, we introduce some more definitions and preliminary lemmas.

Let $G$ be a non-complete connected graph and let $\kappa(G)$ denote the vertex connectivity of $G$.  A \emph{cut} of a connected graph $G$ is a subset $V^{\prime}(G)$ of $V(G)$ such that $G-V^{\prime}(G)$ is disconnected. A \emph{k-cut} is a cut of $k$ elements. By $\mathscr{T}(G):=\{T\subseteq V(G): T$ is a cut of $G, |T|=\kappa(G)\}$, we denote the set of \emph{smallest cuts} of $G$. For $T\in\mathscr{T}(G)$, the union of the vertex sets of at least one but not of all components of $G-T$ is called a \emph{T-fragment} of $G$ or, briefly, a \emph{fragment}. Let $F$ be a $T$-fragment, and let $\overline{F}=V(G)-(F\cup T)$. Clearly, $\overline{F}\neq\emptyset$, and $\overline{F}$ is also a $T$-fragment.

Consider a path $P=x_{0} \dots x_{l}$ in $G$. For any two subsets $A$ and $B$ of $V(G)$, we call $P$ an \emph{A-B path} if $V(P)\cap A=\{x_{0}\}$ and $V(P)\cap B=\{x_{l}\}$. A set of $x$-$A$ paths is called a \emph{x-A fan} if any two of the paths have only $x$ in common. A family of $k$ internally disjoint $(x, Y)$-paths whose terminal vertices are distinct is referred to as a \emph{k-fan from x to Y}.
A \emph{semi-wheel} $SW_{n}$ of order $n$ is defined as $V(SW_{n})=\{x_{0}, x_{1}, \dots, x_{n-1}\}$ and
$E(SW_{n})=\{x_{1}x_{2}, x_{2}x_{3}, \dots, x_{n-2}x_{n-1}, x_{0}x_{2}, x_{0}x_{3}, \dots, x_{0}x_{n-2}\}$, where $x_{0}$ is its \emph{center}. A \emph{semi-prism} $SPr$ is defined as $V(SPr)=\{x, y, z, x^{\prime}, y^{\prime}, z^{\prime}\}$, and $E(SPr)=\{xy, xz, yz, xx^{\prime}, yy^{\prime}, zz^{\prime}\}$. For later use, we denote $S(SW_{n})=\{x_{0}, x_{1}, x_{n-1}\}$ and
$S(SPr)=\{x^{\prime}, y^{\prime}, z^{\prime}\}$.

The following four lemmas appear as Lemmas 3.1 to 3.4 in~\cite{Tsz2}.

\begin{lem}\label{lem1}
Let $G$ be a 3-connected graph, and let $H$ be a 2-connected subgraph of $G$. Suppose $u$ and $v$ are two vertices that lie in different components of $G-V(H)$. Then $\{u, v\}$ is a contractible non-edge.
\end{lem}

\begin{lem}\label{lem2}
Let $G$ be a 3-connected graph, and let $S=\{x, y, z\}$ be a 3-cut of $G$. Consider a component $C$ of $G-S$. Let $F$ be a fan of minimum order among all $c-S$ fans in $G[V(C)\cup S]$ and among all vertices $c$ in $C$. If $V(F)=V(C)\cup S$, then $F^{\prime}:=G[V(F)]-\{xy, yz, zx\}$ is either a semi-wheel or a semi-prism, and $S(F^{\prime})=S$.
\end{lem}

\begin{lem}\label{lem3}
Let $G$ be a 3-connected graph, and let $S$ be a 3-cut of $G$. If $G-S$ consists of $k$ components where $k\geq4$, then $G$ contains at least ${k\choose 2}$ contractible non-edges.
\end{lem}

\begin{lem}\label{lem4}
Let $G$ be a 3-connected graph, and let $S$ be a 3-cut of $G$. Suppose $G-S$ consists of three components, denoted by $C_{1}$, $C_{2}$, $C_{3}$. Consider a $c_{1}-S$ fan $F_{1}$ in $G[V(C_{1})\cup S]$ where $c_{1}\in C_{1}$. If $(V(C_{1})\cup S)-V(F_{1})\neq\emptyset$, then $G$ contains at least two contractible non-edges.
\end{lem}

The following lemma states an elementary property of fragments.

\begin{lem}\label{lem5}\cite{Mader}
Let $F$ be a $T$-fragment and $F^{\prime}$ be a $T^{\prime}$-fragment. If $F\cap F^{\prime}\neq\emptyset$, then $|F\cap T^{\prime}|\geq|\overline{F^{\prime}}\cap T|$.
\end{lem}

\section{Some further Lemmas}
Throughout this section, we adopt the following assumptions and definitions. Let $G$ be a 3-connected graph, and let $S=\{x, y, z\}$ be a 3-cut of $G$. If $C_{i}$ is a component of $G-S$, we denote $F_{i}$ to be a 3-fan of minimum order among all 3-fans from $c_{i}$ to $S$ in $G[V(C_{i})\cup S]$ and among all vertices $c_{i}$ in $C_{i}$. Define $F_{i}^{\prime}=G[V(F_{i})]-\{xy, xz, yz\}$. If $F_{i}^{\prime}$ is a semi-prism, we define $V(F_{i}^{\prime})=\{x_{i}, y_{i}, z_{i}, x, y, z\}$ and $E(F_{i}^{\prime})=\{x_{i}y_{i}, y_{i}z_{i}, z_{i}x_{i}, x_{i}x, y_{i}y, z_{i}z\}$. If $F_{i}^{\prime}$ is a semi-wheel of order $k$, we define $V(F_{i}^{\prime})=\{f(i)_{0}, f(i)_{1}, \dots, f(i)_{k-1}\}$ and $E(F_{i}^{\prime})=\{f(i)_{1}f(i)_{2}, f(i)_{2}f(i)_{3}, \dots, f(i)_{k-2}f(i)_{k-1}, f(i)_{0}f(i)_{2}, f(i)_{0}f(i)_{3}, \dots, f(i)_{0}f(i)_{k-2}\}$.

\begin{lem}\label{lem6}
If $G-S$ consists of three components, then $G$ either has at least two contractible non-edges or is isomorphic to $H_{1}^{\prime}$, $H_{2}^{\prime}$, $H_{4}^{\prime}$, or $H_{8}^{\prime}$ as shown in Fig. \ref{fig1}.
\end{lem}

\begin{proof}
Suppose that $G$ has at most one contractible non-edge. The three components are denoted by $C_{1}, C_{2}, C_{3}$. If $(V(C_{i})\cup S)-V(F_{i})\neq\emptyset$ for some $i\in\{1, 2, 3\}$, then by Lemma~\ref{lem4}, $G$ contains at least two contractible non-edges, a contradiction. So we assume that for $i=1, 2, 3$, $V(C_{i})\cup S=V(F_{i})$. Then by Lemma \ref{lem2}, for $i=1,2,3$, $F_{i}^{\prime}$ is either a semi-prism or a semi-wheel, and $S(F_{i}^{\prime})=\{x, y, z\}$. If $G[S]$ contains at least two edges, then each $G[V(F_{i})]$ is 2-connected, implying that any two vertices in distinct components are contractible non-edges. Thus, $G$ has at least three contractible non-edges, a contradiction. Therefore, we assume that $G[S]$ has at most one edge. Let $f(1)=a$, $f(2)=b$ and $f(3)=u$.

{\bf Case 1.} At most one of $F_{1}^{\prime}$, $F_{2}^{\prime}$ and $F_{3}^{\prime}$ is a semi-wheel of order four.

That is, at least two of them are semi-primes or semi-wheels of order at least five. Without loss of generality, we assume these to be $F_{1}^{\prime}$ and $F_{2}^{\prime}$.
Suppose that $F_{1}^{\prime}$ and $F_{2}^{\prime}$ are semi-wheels of orders $k$ and $l$, respectively.
Regardless of whether the centers of $F_{1}^{\prime}$ and $F_{2}^{\prime}$ are the same, we can have that $G-\{a_{2}, b_{l-2}\}$ and $G-\{a_{k-2}, b_{2}\}$ are 2-connected, implying that $\{a_{2}, b_{l-2}\}$ and $\{a_{k-2}, b_{2}\}$ are contractible by Lemma \ref{lem1}.
Now, we suppose that $F_{1}^{\prime}$ is a semi-prime. If $F_{2}^{\prime}$ is also a semi-prime, then it is obvious that $\{x_{1}, z_{2}\}$ and $\{z_{1}, x_{2}\}$ are contractible. If $F_{2}^{\prime}$ is a semi-wheel of order $l$, we assume $b_{0}=x$ and $b_{1}=y$ without loss of generality. Then we find that $\{x_{1}, b_{l-2}\}$ and $\{z_{1}, b_{2}\}$ are both contractible similarly. In each case, there is a contradiction.

{\bf Case 2.} At least two of $F_{1}^{\prime}$, $F_{2}^{\prime}$ and $F_{3}^{\prime}$ are semi-wheels of order four.

Without loss of generality, we assume that $F_{1}^{\prime}$ and $F_{2}^{\prime}$ are the semi-wheels of order four. Let $F_{3}^{\prime}$ be a semi-prime. If $G[S]$ has one edge, without loss of generality assume it to be $xy$, then $\{a_{2}, x_{3}\}$ and $\{b_{2}, x_{3}\}$ are contractible non-edges of $G$. Therefore, $G[S]$ has no edges, and thus $G\cong H_{8}^{\prime}$.
Let $F_{3}^{\prime}$ be a semi-wheel of order $r$. Without loss of generality, we assume that $u_{0}=x$ and $u_{1}=y$.
If $r\geq6$, then $yb_{2}zu_{r-2}xu_{2}y$ is 2-connected, implying that $\{a_{2}, u_{3}\}$ is contractible by Lemma \ref{lem1}. Similarly, we have that $\{b_{2}, u_{3}\}$ is contractible, then $G$ has at least two contractile non-edges, a contradiction.
If $r=4$, then $G\cong H_{1}^{\prime}$ when $G[S]$ has no edges, and  $G\cong H_{2}^{\prime}$ when $G[S]$ contains one edge.
Consider the case $r=5$. If $G[S]$ contain one edge, then we can get that $G$ has at least two contractile non-edges. So $G[S]$ contains no edges, implying that $G\cong H_{4}^{\prime}$.
\end{proof}

In the following four lemmas, we assume that $G-S$ consists of two components, denoted by $C_{1}$ and $C_{2}$. The three paths in $F_{i}$ starting from $c_{i}$ to $x, y, z$ are denoted by $P_{x}^{i}, P_{y}^{i}, P_{z}^{i}$ respectively. Let $f(1)=a$ and $f(2)=b$.

\begin{lem}\label{lem7}
If $V(C_{1})\cup S\neq V(F_{1})$ and $V(C_{2})\cup S\neq V(F_{2})$, then $G$ has at least two contractible non-edges.
\end{lem}

\begin{proof}
Clearly, $G[V(F_{1})\cup V(F_{2})]$ is 2-connected. Then, for each $u\in (V(C_{1})\cup S)-V(F_{1})$ and $v\in (V(C_{2})\cup S)-V(F_{2})$, the non-edge $\{u, v\}$ is contractible by Lemma~\ref{lem1}. This implies that the lemma holds if $|(V(C_{1})\cup S)-V(F_{1})|\geq2$ or $|(V(C_{2})\cup S)-V(F_{2})|\geq2$. So we assume that $|(V(C_{1})\cup S)-V(F_{1})|=|(V(C_{2})\cup S)-V(F_{2})|=1$. Let $(V(C_{1})\cup S)-V(F_{1})=\{a\}$ and $(V(C_{2})\cup S)-V(F_{2})=\{b\}$. Note that $\{a, b\}$ is a contractible non-edge of $G$.
If $\{c_{1}, b\}$ is not contractible, then there exists $T\in\mathscr{T}(G)$ such that $T\supset\{c_{1}, b\}$. Let $B_{1}$ be a $T$-fragment and let $B_{2}=G-T-B_{1}$. If $S\cap B_{1}=\emptyset$, then Lemma \ref{lem5} assures us $C_{1}\cap B_{1}=C_{2}\cap B_{1}=\emptyset$, and thus, $B_{1}=\emptyset$, a contradiction. So $S\cap B_{1}\neq\emptyset$. Similarly, $S\cap B_{2}\neq\emptyset$. It follows $|C_{2}\cap T|=2$, and hence, $S\cap T=\emptyset$ and $C_{1}\cap T=\{c_{1}\}$. Without loss of generality, we assume $S\cap B_{1}=\{x\}$.
By Lemma \ref{lem5}, we have $C_{1}\cap B_{1}=\emptyset$. Thus, $a\in C_{1}\cap B_{2}$. This implies that $yP_{y}^{1}c_{1}P_{z}^{1}zP_{z}^{2}c_{2}P_{y}^{2}y$ is 2-connected, and $x$ and $a$ lie in different components. By Lemma \ref{lem1}, $\{x, a\}$ is a contractible non-edge. It follows that either $\{c_{1}, b\}$ or $\{x, a\}$ is a contractible non-edge of $G$, and hence, the lemma holds.
\end{proof}

\begin{lem}\label{lem8}
If $V(C_{1})\cup S=V(F_{1})$ and $V(C_{2})\cup S=V(F_{2})$, then $G$ either has at least two contractible non-edges or is isomorphic to $K_{5}^{-}$, $H_{1}$, $H_{2}$, $H_{10}^{\prime}$, or $W_{n}$ for $n\geq 5$.
\end{lem}

\begin{proof}
By Lemma~\ref{lem2}, we have that for $i=1, 2$, $F_{i}^{\prime}$ is either a semi-prime or a semi-wheel, and $S(F_{i}^{\prime})=S$. Suppose that $G$ has at most one contractible non-edge. If $|C_{1}|=|C_{2}|=1$, then $G$ is isomorphic to $W_{5}$ or $K_{5}^{-}$. We assume that $|C_{1}|\geq2$ or $|C_{2}|\geq2$.
If $G[S]\cong K_{3}$, then for each $u\in V(C_{1})$ and $v\in V(C_{2})$, the non-edge $\{u, v\}$ is contractible, and thus, $G$ has at least two contractible non-edges, a contradiction. So we assume that $G[S]$ has at most two edges.

{\bf Case 1.} $F_{1}^{\prime}$ or $F_{2}^{\prime}$ is a semi-prime.

Without loss of generality, assume that $F_{1}^{\prime}$ is a semi-prime.
If $F_{2}^{\prime}$ is also a semi-prime, then we see that $G[S]$ has two edges. If follows that $G$ has at least two contractible non-edges, a contradiction.
If $F_{2}^{\prime}$ is a semi-wheel of order $l\geq5$, then we may assume that $b_{0}=x$ and $b_{1}=y$ without loss of generality. If $yz\notin E(G)$, then $\{xy, xz\}\subset E(G)$, implying that $G-x_{1}-b_{2}$ and $G-x_{1}-b_{l-2}$ are both 2-connected. If $yz\in E(G)$, then $G-y_{1}-b_{l-2}$ and $G-z_{1}-b_{2}$ are both 2-connected. By Lemma \ref{lem1}, $G$ has at least two contractible non-edges, a contradiction.
Therefore, $F_{2}^{\prime}$ is a semi-wheel of order four, and then, $G\cong H_{2}$.

{\bf Case 2.} One of $F_{1}^{\prime}$ and $F_{2}^{\prime}$ is a semi-wheel of order four, and the other is a semi-wheel of order $l\geq5$.

Without loss of generality, we assume that $F_{1}^{\prime}$ is a semi-wheel of order four, and $b_{0}=x$, $b_{1}=y$.
If $yz\notin E(G)$, then $\{xy, xz\}\subset E(G)$, and then $G\cong W_{n}$ where $n\geq6$. Consider the case when $yz\in E(G)$. If $l=5$, then $G\cong H_{1}$ when $G[S]\cong P_{3}$, and $G\cong H_{10}^{\prime}$ when $G[S]$ contains one edge.
If $l\geq6$, then we see that $xb_{2}yzb_{l-2}x$ is 2-connected, implying that for $3\leq j\leq l-3$, $\{a_{2}, b_{j}\}$ is contractible by Lemma \ref{lem1}. Therefore, $l=6$. If $xy\in E(G)$ or $xz\in E(G)$, then we see that $\{a_{2}, b_{2}\}$ or $\{a_{2}, b_{4}\}$ is contractible, a contradiction. So $xy\notin E(G)$ and $xz\notin E(G)$, implying that $G\cong H_{2}$.

{\bf Case 3.} Both $F_{1}^{\prime}$ and $F_{2}^{\prime}$ are semi-wheels of order at least five.

Let the order of $F_{1}^{\prime}$ and $F_{2}^{\prime}$ be $k$ and $l$, respectively. Suppose that they have a common center $x$, and $a_{1}=b_{1}=y$ without loss of generality. If $yz\notin E(G)$, then $G\cong W_{n}$ where $n\geq7$. If $yz\in E(G)$, then we see that $xa_{k-2}zyb_{2}x$ and $xa_{2}yzb_{l-2}x$ is 2-connected. It follows that both $\{a_{2}, b_{l-2}\}$ and $\{a_{k-2}, b_{2}\}$ are contractible non-edges by Lemma \ref{lem1}.
Hence, they have different centers. Without loss of generality, we assume that $a_{0}=b_{1}=x$ and $a_{1}=b_{0}=y$. Since $d_{G}(z)\geq3$, $xz\in E(G)$ or $yz\in E(G)$. If $xz\in E(G)$, then $xa_{2}yb_{l-2}zx$ is 2-connected, implying that for any $3\leq i\leq k-2$ and $2\leq j\leq l-3$, $\{a_{i}, b_{j}\}$ are contractible by Lemma \ref{lem1}. It follows $k=l=5$. Furthermore, $xy\notin E(G)$ and $yz\notin E(G)$. Otherwise, either $\{a_{2}, b_{2}\}$ or $\{a_{2}, b_{3}\}$ is contractible, and thus, $G$ has at least two contractible non-edges, a contradiction. Consequently, we find $G\cong H_{2}$. If $yz\in E(G)$, we can similarly conclude that $G\cong H_{2}$.
\end{proof}

\begin{lem}\label{lem9}
If $V(C_{1})\cup S=V(F_{1})$, $V(C_{2})\cup S\neq V(F_{2})$ and $|V(C_{1})|=2$, then $G$ either has at least two contractible non-edges, or is isomorphic to $H_{5}^{\prime}$, $H_{8}^{\prime}$, $H_{9}^{\prime}$, $H_{2}$ or $H_{4}$.
\end{lem}

\begin{proof}
Suppose that $G$ has at most one contractible non-edge. By Lemma~\ref{lem2}, $F_{1}^{\prime}$ is a semi-wheel of order five. Denote $a_{0}=x$, $a_{1}=y$, and $a_{4}=z$ without loss of generality. Since $G$ is 3-connected, $d_{G}(y)\geq3$ and $d_{G}(z)\geq3$, and then, $N_{G}(y)-N_{F_{1}\cup F_{2}}(y)\neq\emptyset$ and $N_{G}(z)-N_{F_{1}\cup F_{2}}(z)\neq\emptyset$. Let $W=(V(C_{2})\cup S)-V(F_{2})$.

\noindent{\bf Claim 1.} $N_{G}(y)-N_{F_{1}\cup F_{2}}(y)\subseteq W$ and $N_{G}(z)-N_{F_{1}\cup F_{2}}(z)\subseteq W$.

\begin{proof}
We only show that $N_{G}(y)-N_{F_{1}\cup F_{2}}(y)\subseteq W$. Clearly, $y$ has only one neighbour in $V(P_{y}^{2})$, otherwise, it would contradict the minimality of $F_{2}$.
Next, we show that there does not exist a $y^{\prime}$ such that $y^{\prime}\in N_{G}(y)\cap ((V(P_{x}^{2})\cup V(P_{z}^{2}))-\{c_{2}\})$. By contradiction, without loss of generality, assume that there exists $y^{\prime}\in V(P_{z}^{2})-\{c_{2}\}$ such that $yy^{\prime}\in E(G)$. Since $xa_{3}zy^{\prime}yP_{y}^{2}c_{2}P_{x}^{2}x$ is 2-connected, we see that $|W|=1$ and $c_{2}y^{\prime}\in E(G)$, otherwise, $G$ has at least two contractible non-edges by Lemma \ref{lem1}. Let $W=\{w\}$. Then $\{a_{2}, w\}$ is a contractible non-edge of $G$. We observe that for each $s\in\{x, y, z\}$, there exists $u\in V(P_{s}^{2})-\{c_{2}\}$ such that $uw\in E(G)$. Otherwise, $F_{1}\cup F_{2}-(V(P_{s}^{2})-\{c_{2}\})$ is 2-connected, and $s$ and $w$ lie in different components, implying that $\{s, w\}$ is a contractible non-edge by Lemma \ref{lem1}, a contradiction.

Since $\{a_{2}, c_{2}\}$ is not contractible, there is a 3-cut $T$ such that $T\supseteq\{a_{3}, c_{2}\}$. Let $B_{1}$ be a $T$-fragment and let $B_{2}=G-T-B_{1}$. Then $S\cap B_{1}\neq\emptyset$ and $S\cap B_{2}\neq\emptyset$. Without loss of generality, we assume $|S\cap B_{1}|=1$. If $C_{2}\cap T=\{c_{2}\}$, then for some $s\in\{x, y, z\}$, there does not exist $u\in V(P_{s}^{2})-\{c_{2}\}$ such that $uw\in E(G)$, a contradiction. Hence, $|C_{2}\cap T|=2$, and then, $S\cap T=\emptyset$ and $C_{1}\cap T=\{a_{2}\}$. By Lemma \ref{lem5}, $C_{1}\cap B_{1}=\emptyset$. It follows $C_{1}\cap B_{2}=\{a_{3}\}$, and thus, $S\cap B_{1}=\{y\}$. Since $y^{\prime}\in N_{G}(y)$, $y^{\prime}\in C_{2}\cap(B_{1}\cup T)$. If $y^{\prime}\in C_{2}\cap B_{1}$, then $(C_{2}\cap T)-\{c_{2}\}\in V(P_{z}^{2})$. It follows $w\in C_{2}\cap(B_{1}\cup B_{2})$, and thus, there does not exist $u\in V(P_{x}^{2})-\{c_{2}\}$ or $V(P_{y}^{2})-\{c_{2}\}$ such that $uw\in E(G)$, a contradiction. If $y^{\prime}\in C_{2}\cap T$, we can similarly obtain a contradiction.
\end{proof}

By Claim 1, we have that $G[S]$ has no edges.

\noindent{\bf Claim 2.} There exists a $y-z$ path that is internally disjoint from $F_{1}\cup F_{2}$.

\begin{proof}
By contradiction, suppose that every $y-z$ path is not internally disjoint from $F_{1}\cup F_{2}$. Let $y^{\prime}\in N_{G}(y)-N_{F_{1}\cup F_{2}}(y)$ and $z^{\prime}\in N_{G}(z)-N_{F_{1}\cup F_{2}}(z)$. Then $y^{\prime}\neq z^{\prime}$ and $y^{\prime}z^{\prime}\notin E(G)$.
Note that $G[V(F_{1})\cup V(F_{2})]$ is 2-connected and $G[W]$ is disconnected. If $W$ has at least three parts or some part contains at least two vertices, then $G$ has at least two contractible non-edges by Lemma~\ref{lem1}. So $W=\{y^{\prime}, z^{\prime}\}$ and $\{y^{\prime}, z^{\prime}\}$ is a contractible non-edge.
If there exists $x^{\prime}\in V(P_{x}^{2})-\{c_{2}\}$ such that $y^{\prime}x^{\prime}\in E(G)$, then $xa_{3}zP_{z}^{2}c_{2}P_{y}^{2}yy^{\prime}x^{\prime}P_{x}^{2}x$ is 2-connected, implying that $\{a_{2}, z^{\prime}\}$ is contractible by Lemma \ref{lem1}.
If there is no $x^{\prime}\in V(P_{x}^{2})-\{c_{2}\}$ such that $y^{\prime}x^{\prime}\in E(G)$, then $a_{2}a_{3}zP_{z}^{2}c_{2}P_{y}^{2}ya_{2}$ is 2-connected, and thus, $\{x, y^{\prime}\}$ is contractible. In either case, $G$ has at least two contractile non-edges, a contradiction.
\end{proof}

Let $P=u_{0}u_{1} \dots u_{m}u_{m+1}$ be a shortest such path, where $u_{0}=y$ and $u_{m+1}=z$. Observe that $m\geq1$.
Since $xa_{3}zPyP_{y}^{2}c_{2}P_{x}^{2}x$ is a 2-connected subgraph of $G$, for each $u\in(V(P_{z}^{2})-\{z, c_{2}\})\cup(W-V(P))$, non-edge $\{a_{2}, u\}$ is contractible.
Furthermore, non-edge $\{a_{3}, u\}$ is contractible for each $u\in(V(P_{y}^{2})-\{y, c_{2}\})\cup(W-V(P))$ because $xa_{2}yPzP_{z}^{2}c_{2}P_{x}^{2}x$ is 2-connected.
Consequently, $|V(P_{y}^{2})|\leq3$, $|V(P_{z}^{2})|\leq3$, and $W=V(P)-\{y, z\}$. It follows $d_{G}(y)=d_{G}(z)=3$.

If there does not exist $u_{i}$ such that $x^{\prime}u_{i}\in E(G)$, where $1\leq i\leq m$ and $x^{\prime}\in V(P_{x}^{2})-\{c_{2}\}$, then $\{x, u_{i}\}$, where $1\leq i\leq m$, are contractible non-edges since $ya_{2}a_{3}zP_{z}^{2}c_{2}P_{y}^{2}y$ is 2-connected.
If there exists $x^{\prime}\in V(P_{x}^{2})-\{c_{2}\}$ such that $x^{\prime}u_{i}\in E(G)$ for $1<i\leq m$, then $xa_{2}yP_{y}^{2}c_{2}P_{z}^{2}zu_{m}Pu_{i}x^{\prime}P_{x}^{2}x$ is 2-connected, implying that $\{a_{3}, u_{i}\}$ where $1\leq i\leq m-1$ is a contractible non-edge. Similarly, we have that $\{a_{2}, u_{j}\}$ where $2\leq j\leq m$ is contractible if $x^{\prime}u_{i}\in E(G)$ for $1\leq i<m$, because $xa_{3}zP_{z}^{2}c_{2}P_{y}^{2}yu_{1}Pu_{i}x^{\prime}P_{x}^{2}x$ is 2-connected.
Therefore, $m\leq2$. Furthermore, if $m=2$, we have that $\{a_{3}, u_{1}\}$ or $\{a_{2}, u_{2}\}$ is the unique contractible non-edge of $G$.

\noindent{\bf Claim 3.} If $zc_{2}\notin E(G)$, then $G\cong H_{4}$.

\begin{proof}
Let $V(P_{z}^{2})=\{z, t, c_{2}\}$. Then $\{a_{2}, t\}$ is a contractible non-edge of $G$. From the above statement, it follows that $yc_{2}\in E(G)$ and $m=1$.
Moreover, we see that there does not exist $x^{\prime}\in V(P_{x}^{2})-\{c_{2}\}$ such that $x^{\prime}t\in E(G)$. Otherwise, $xa_{2}yu_{1}ztx^{\prime}P_{x}^{2}x$ is a 2-connected subgraph, and by Lemma \ref{lem1}, $\{a_{3}, c_{2}\}$ is contractible, a contradiction.
So $N_{G}(t)=\{z, c_{2}, u_{1}\}$.
If there exist $x^{\prime\prime}\in V(P_{x}^{2})-\{x\}$ such that $x^{\prime\prime}u_{1}\in E(G)$, then $ya_{2}a_{3}zu_{1}x^{\prime\prime}P_{x}^{2}c_{2}y$ is 2-connected, implying that $\{x, t\}$ is contractible by Lemma \ref{lem1}, a contradiction. So $N_{G}(u_{1})=\{y, z, t, x\}$. Hence, $xc_{2}\in E(G)$, and then, $G\cong H_{4}$.
\end{proof}

Similarly, we can deduce that $G\cong H_{4}$ if $yc_{2}\notin E(G)$. In the following, we assume that $yc_{2}\in E(G)$ and $zc_{2}\in E(G)$. Consider the case $m=2$. Without loss of generality, we assume that $x^{\prime}u_{1}\in E(G)$ for some $x^{\prime}\in V(P_{x}^{2})-\{c_{2}\}$. Thus, $\{a_{2}, u_{2}\}$ is the unique contractible non-edge of $G$, and $N_{G}(u_{2})=\{z, u_{1}, c_{2}\}$. If $x^{\prime}\neq x$, then $ya_{2}a_{3}zc_{2}P_{x}^{2}x^{\prime}u_{1}y$ is 2-connected, and thus, $\{x, u_{2}\}$ is contractible, a contradiction. So $xc_{2}\in E(G)$. Similarly, we have that $\{x, u_{2}\}$ is contractible if $c_{2}u_{1}\in E(G)$. So $c_{2}u_{1}\notin E(G)$, and then, $G\cong H_{4}$.

Let $m=1$. If there exists two vertices $\{x_{1}, x_{2}\}\subseteq V(P_{x}^{2})-\{c_{2}\}$ where $d_{G}(x, x_{1})<d_{G}(x, x_{2})$ such that $x_{1}u_{1}\in E(G)$ and $x_{2}u_{1}\in E(G)$, then we see that $\{a_{2}, x_{2}\}$ and $\{a_{3}, x_{2}\}$ are contractible, a contradiction. It follows $|V(P_{x}^{2})|\leq3$.
If $|V(P_{x}^{2})|=3$, then we let $V(P_{x}^{2})=\{x, x^{\prime}, c_{2}\}$. Then we have that $N_{G}(x^{\prime})=\{x, c_{2}, u_{1}\}$ and $N_{G}(x)=\{a_{2}, a_{3}, x^{\prime}\}$. It follows that $G$ is isomorphic to $H_{8}^{\prime}$ or $H_{9}^{\prime}$ depending on whether the edge $c_{2}u_{1}$ exists.
If $|V(P_{x}^{2})|=2$, then $c_{2}u_{1}\in E(G)$ since $C_{2}$ is connected. Hence, $G\cong H_{5}^{\prime}$ if $xu_{1}\in E(G)$, and $G\cong H_{2}$ if $xu_{1}\notin E(G)$.
\end{proof}

\begin{lem}\label{lem10}
If $V(C_{1})\cup S=V(F_{1})$, $V(C_{2})\cup S\neq V(F_{2})$ and $|V(C_{1})|\geq3$, then $G$ either has at least two contractible non-edges or is isomorphic to $H_{5}^{\prime}$, $H_{8}^{\prime}$, $H_{9}^{\prime}$, $H_{2}$ or $H_{4}$.
\end{lem}

\begin{proof}
Suppose that $G$ has at most one contractible non-edge. By Lemma~\ref{lem2}, $F_{1}^{\prime}$ is a semi-prime or a semi-wheel of order at least six.
If $F_{1}^{\prime}$ is a semi-wheel of order $k$, we assume that $a_{0}=x$ and $a_{1}=y$ without loss of generality. Then we observe that $xa_{2}yP_{y}^{2}c_{2}P_{z}^{2}za_{k-2}x$ is 2-connected, implying that $\{a_{i}, v\}$ is contractible by Lemma~\ref{lem1}, where $3\leq i\leq k-3$ and $v\in W\cup V(P_{x}^{2})-\{x, c_{2}\}$. It follows that $k=6$, $|W|=1$ and $xc_{2}\in E(G)$. Let $W=\{w\}$. Note that $\{a_{3}, w\}$ is a contractible non-edge of $G$. If $xw\notin E(G)$, then $a_{2}a_{3}a_{4}zP_{z}^{2}c_{2}P_{y}^{2}ya_{2}$ is 2-connected, implying that $\{x, w\}$ is contractible by Lemma \ref{lem1}, a contradiction. So $xw\in E(G)$. If there exists $y^{\prime}\in N_{G}(y)\cap(V(P_{z}^{2})-\{c_{2}\})$, then $xa_{4}zP_{z}^{2}y^{\prime}yP_{y}^{2}c_{2}x$ is 2-connected, and thus, $\{a_{2}, w\}$ is contractible by Lemma \ref{lem1}, a contradiction. So $N_{G}(y)\cap (V(P_{z}^{2})-\{c_{2}\})=\emptyset$. Similarly, we have $yx\notin E(G)$. Therefore, $d_{G}(y)=3$ and $w\in N_{G}(y)$. Similarly, $d_{G}(z)=3$ and $w\in N_{G}(z)$. Then we see that $xa_{2}ywza_{4}x$ is 2-connected, implying that $\{a_{3}, c_{2}\}$ is contractible, a contradiction.
So $F_{1}^{\prime}$ is a semi-prime. Then we can transform the problem to Lemma \ref{lem9} by considering the 3-cut $\{x_{1}, y, z\}$, and thus, $G$ is isomorphic to $H_{5}^{\prime}$, $H_{8}^{\prime}$, $H_{9}^{\prime}$, $H_{2}$ or $H_{4}$.
\end{proof}

\section{Proof of Theorem~\ref{thm1}}
In this section, we give a proof of Theorem~\ref{thm1}.
Let $G$ be a 3-connected graph that contains exactly one contractible non-edge. If $G$ is 4-connected, then all non-edges are contractible, which implies that $G$ must contain exactly one non-edge. Therefore, $G\cong K_{n}^{-}$ where $n\geq 6$. Consequently, we may assume that $S=\{x, y, z\}$ is a 3-cut of $G$. By Lemmas \ref{lem3} and \ref{lem6}, $G-S$ has two components, denoted by $C_{1}$ and $C_{2}$. By Lemma \ref{lem7}, either $V(C_{1})\cup S=V(F_{1})$ or $V(C_{2})\cup S=V(F_{2})$. Without loss of generality, we assume $V(C_{1})\cup S=V(F_{1})$. If $V(C_{2})\cup S=V(F_{2})$, then by Lemma \ref{lem8}, $G$ is isomorphic to $K_{5}^{-}$, $H_{1}$ or $H_{2}$. Hence, we may assume $V(C_{2})\cup S\neq V(F_{2})$. Let $W=(V(C_{2})\cup S)-V(F_{2})$. If $|V(C_{1})|\geq2$, then by Lemmas \ref{lem9} and \ref{lem10}, $G$ is isomorphic to $H_{2}$ or $H_{4}$. Consequently, we may further assume that:

(1) For every 3-cut $T$ of $G$, $G-T$ has exactly two components, one of which consists of exactly one vertex, denoted by $x_{T}$, and the another component denoted by $C_{T}$.

(2) For any fan $F$ of minimum order among all $c-T$ fans in $G[V(C_{T})\cup T]$ and among all vertices $c$ in $C_{T}$, $(V(C_{T})\cup T)-V(F)\neq\emptyset$.

(3) If $\{u, v\}$ is a non-contractible non-edge, then there exists a vertex $t$ such that $t\in N_{G}(u)\cap N_{G}(v)$ and $d_{G}(t)=3$. We refer to this vertex $t$ as the \emph{``triad of u and v"}.

By (3), we have $2\leq|V(C_{S})|\leq7$.

\noindent{\bf Claim 1.} Every triangle of $G$ has at most one vertex of degree three.

\begin{proof}
Suppose that there exists a triangle $abc$ such that $d_{G}(a)=d_{G}(b)=3$. Let $N_{G}(a)=\{b, c, a_{1}\}$ and $N_{G}(b)=\{a, c, b_{1}\}$. Since $G$ is 3-connected, we have $a_{1}\neq b_{1}$, which implies that $R=\{a_{1}, b_{1}, c\}$ is a 3-cut of $G$. By (1), $G-R$ has a component consisting of a vertex $x_{R}$, and $G[\{a, b\}]$ forms another component. However, this contradicts (2).
\end{proof}

\noindent{\bf Claim 2.} $|V(C_{S})|\neq2$.

\begin{proof}
Suppose $|V(C_{S})|=2$. Let $V(C_{S})=\{a, b\}$. Since $C_{S}$ is connected, $ab\in E(G)$. By (3), we have that $S$ has one vertex of degree three. Without loss of generality, we assume $d_{G}(x)=3$. By Claim 1, we have $N_{G}(x)=\{x_{S}, a, b\}$, implying that $N_{G}(a)=\{x, y, z, b\}$ and $N_{G}(b)=\{x, y, z, a\}$. Then we see that $G\cong H_{3}^{\prime}$ when $yz\in E(G)$ and $G\cong H_{2}^{\prime}$ when $yz\notin E(G)$, a contradiction.
\end{proof}

\noindent{\bf Claim 3.} $|V(C_{S})|\neq4$.

\begin{proof}
Suppose $|V(C_{S})|=4$ and let $V(C_{S})=\{v_{1}, v_{2}, v_{3}, v_{4}\}$. According to (3), every 3-cut of $G$ contains at least two vertices of degree three. By Claim 1, every 3-cut is an independent set. Without loss of generality, assume $d_{G}(x)=d_{G}(y)=3$. If $N_{G}(x)=N_{G}(y)$, then $G-N_{G}(x)$ has at least three components, which contradicts (1). Therefore, $N_{G}(x)\neq N_{G}(y)$.

If $|N_{G}(x)\cap N_{G}(y)|=2$, we may assume that $N_{G}(x)=\{x_{S}, v_{1}, v_{2}\}$ and $N_{G}(y)=\{x_{S}, v_{2}, v_{3}\}$ without loss of generality. Then $\{x_{S}, v_{1}, v_{2}\}$ and $\{x_{S}, v_{2}, v_{3}\}$ are both 3-cuts of $G$, which implies that $v_{1}v_{2}\notin E(G)$ and $v_{2}v_{3}\notin E(G)$. Since $C_{S}$ is connected, $v_{2}v_{4}\in E(G)$ and we may assume $v_{3}v_{4}\in E(G)$ without loss of generality. Then we deduce that $v_{1}v_{3}\in E(G)$, and either $d_{G}(v_{1})=3$ or $d_{G}(v_{3})=3$. Otherwise, both $\{x, v_{3}\}$ and $\{y, v_{1}\}$ are contractible non-edges, a contradiction. If $d_{G}(v_{3})\neq3$, then $N_{G}(v_{3})=\{y, z, v_{1}, v_{4}\}$ and $d_{G}(v_{1})=3$. However, in this case, $N_{G}(v_{1})$ is not an independent set, a contradiction. Thus, $d_{G}(v_{3})=3$, which implies $v_{1}v_{4}\notin E(G)$. Then we have that $N_{G}(v_{1})=\{x, z, v_{3}\}$ and $N_{G}(v_{4})=\{z, v_{2}, v_{3}\}$. This implies $zv_{2}\notin E(G)$, and then $G\cong H_{6}^{\prime}$, a contradiction.

If $|N_{G}(x)\cap N_{G}(y)|=1$, we may assume that $N_{G}(x)=\{x_{S}, v_{1}, v_{2}\}$ and $N_{G}(y)=\{x_{S}, v_{3}, v_{4}\}$ without loss of generality. Then $v_{1}v_{2}\notin E(G)$ and $v_{3}v_{4}\notin E(G)$. Note that $d_{G}(z)\neq3$, otherwise, as in the above discussion, we can get a contradiction. Without loss of generality, we assume $\{zv_{2}, zv_{3}, zv_{4}\}\subseteq E(G)$. Since either $\{x, v_{3}\}$ or $\{x, v_{4}\}$ is non-contractible, there exists a triad of $x$ and $v_{3}$ or a triad of $x$ and $v_{4}$ by (3). This vertex must be $v_{1}$. Then $N_{G}(v_{1})$ is an independent set, and then $N_{G}(v_{1})=\{x, v_{3}, v_{4}\}$. Thus, we observe that there is no vertex that is a triad of $y$ and $v_{2}$, which implies that $\{y, v_{2}\}$ is a contractible non-edge. It follows that $\{z, v_{1}\}$ is non-contractible. Without loss of generality, we assume $v_{4}$ is a triad of $z$ and $v_{1}$. Since $d_{G}(v_{2})\geq3$, $v_{2}v_{3}\in E(G)$. However, $N_{G}(v_{2})=\{x, z, v_{3}\}$ is not an independent set, a contradiction.
\end{proof}

Similarly, we can deduce that  $|C_{S}|\neq 5, 6,7$. Consequently, $|C_{S}|=3$. Let $V(C_{S})=\{a, b, c\}$. By (3), we see that $S$ contains one vertex of degree three. Without loss of generality, we assume $d_{G}(x)=3$. By Claim 1, we may assume $N_{G}(x)=\{x_{S}, a, b\}$. If for any $u\in\{y, z, a, b\}$, $d(u)\geq4$, then $\{x_{S}, c\}$ and $\{x, c\}$ are contractible non-edges by (3), a contradiction.

So we may assume $d_{G}(y)=3$ without loss of generality. Then we see that $yz\notin E(G)$ by Claim 1 and $N_{G}(x)\neq N_{G}(y)$. Without loss of generality, we assume $N_{G}(y)=\{x_{S}, b, c\}$. Since $C_{S}$ is connected, either $ba\in E(G)$ or $bc\in E(G)$. We assume $ba\in E(G)$ without loss of generality. Then Claim 1 assures us that $d_{G}(a)\geq4$ and $d_{G}(b)\geq4$. This implies that $N_{G}(a)=\{x, z, b, c\}$ and there is no vertex that is a triad of $x$ and $c$. By (3), we have that $\{x, c\}$ is a contractible non-edge of $G$. Hence, $\{y, a\}$ is not contractible and $c$ must be a triad of $y$ and $a$. It follows $d_{G}(c)=3$. If $zb\notin E(G)$, then by (3), there exists a vertex that is a triad of $z$ and $b$. However, no such vertex exists. Therefore, $zb\in E(G)$. Similarly, $zc\in E(G)$. Then we have $G\cong H_{3}$.
\hfill\qedsymbol

\section*{Declaration of competing interest}
The authors declare that they have no conflicts of interest to this work.

\section*{Data availability}
No data was used for the research described in the article.

\end{document}